\documentclass[12pt]{article} 

\usepackage{amsmath}
\usepackage{amsthm}
\usepackage{amsfonts}
\usepackage{mathrsfs}
\usepackage{stmaryrd}
\usepackage{setspace}
\usepackage{fullpage}
\usepackage{amssymb}
\usepackage{breqn}
\usepackage{enumitem}
\usepackage{bbold} 
\usepackage{authblk}
\usepackage{comment}
\usepackage{hyperref}
\usepackage{pgf,tikz}
\usepackage{graphicx}
\usepackage{colonequals}
\usetikzlibrary{decorations.pathreplacing,arrows}
\tikzstyle{vertex}=[circle,draw=black,fill=black,inner sep=0,minimum size=3pt,text=white,font=\footnotesize]

\newtheorem*{thm*}{Theorem}
\newtheorem{thm}{Theorem}

\newtheorem{lemma}[thm]{Lemma}
\newtheorem{conjecture}[thm]{Conjecture}

\newtheorem{prop}[thm]{Proposition}

\newtheorem{corollary}[thm]{Corollary}

\newtheorem*{proposition*}{Proposition}

\newcommand\cF{{\mathcal F}}
\newcommand\cG{{\mathcal G}}
\newcommand\cH{{\mathcal H}}

\newcommand\cM{{\mathcal M}}

\newcommand\cS{{\mathcal S}}

\newcommand{\ignore}[1]{}

\title{The covering lemma and $q$-analogues of extremal set theory problems}

\author{D\'aniel Gerbner\thanks{Research supported by the J\'anos Bolyai Research Fellowship of the Hungarian Academy of Sciences and the National Research, Development and Innovation Office -- NKFIH under the grants K 116769, KH 130371 and SNN 129364.}
\\
\small Alfr\'ed R\'enyi Institute of Mathematics, Hungarian Academy of Sciences\\
\small P.O.B. 127, Budapest H-1364, Hungary.\\
\small \texttt{gerbner@renyi.hu}}

\begin{document}

\maketitle

\begin{abstract} We prove a general lemma (inspired by a lemma of Holroyd and Talbot) about the connection of the largest cardinalities (or weight) of structures satisfying some hereditary property and substructures satisfying the same hereditary property. We use it to show how results concerning forbidden subposet problems in the Boolean poset imply analogous results in the poset of subspaces of a finite vector space. We also study generalized forbidden subposet problems in the poset of subspaces.
    
\end{abstract}

\vspace{4mm}
	
	
	\noindent
	{\bf AMS Subj.\ Class.\ (2010)}: 06A07, 05D05

\section{Introduction}

One of the most basic question in extremal finite set theory is the following. Given a property of families of subsets of a finite set, what is the largest family satisfying it? Sperner \cite{S} showed that if the property is that no member of the family contains another member (in other words: the family is an \textit{antichain}), the answer is $\binom{n}{\lfloor n/2\rfloor}$. It is sharp, as shown by the family of all the $\lfloor n/2\rfloor$-element subsets.

Our underlying set is $[n]:=\{1,2,\dots,n\}$. We denote the family of all of its subsets by $2^{[n]}$. The family of $i$-element subsets of $[n]$ is called \textit{level $i$} and is denoted by $\binom{[n]}{i}$. Let $\Sigma(n,k)$ denote the cardinality of the largest $k$ levels (i.e. the middle $k$ levels) of $2^{[n]}$. More precisely, $\Sigma(n,k)=\sum_{i=1}^k\binom{n}{\lfloor\frac{n-k}{2}\rfloor+i}$. We say that a family is a \textit{chain} if its members pairwise contain each other, and is a \textit{full chain} if it is a chain and has $n+1$ members (thus one from each level). The chain of $k$ elements is said to have \textit{length $k$} and is denoted by $P_k$. 

To generalize Sperner's theorem, Katona and Tarj\'an \cite{KT} initiated the study of properties given by forbidding inclusion patterns. More precisely, let $P$ be a finite poset. We say that a family $\cF\subset 2^{[n]}$ (weakly) contains $P$ if there is an injection $f:P\rightarrow \cF$ such that if $x<_P y$, then $f(x)\subset f(y)$. Let $La(n,P)$ denote the size of the largest $P$-free family $\cF\subset 2^{[n]}$.

Let us denote by $e(P)$ the largest integer $m$ such that for any $n$, any family $\cF\subseteq 2^{[n]}$ consisting of $m$ \textit{consecutive} levels is $P$-free. Every result in this area suggests that the following might hold.

\begin{conjecture}\label{conj}
For any integer $n$ and poset $P$, we have
$La(n,P)=(1+o(1))\Sigma(n,e(P))=(e(P)+o(1))\binom{n}{\lfloor n/2\rfloor}$.
\end{conjecture}

This was first stated as a conjecture by Griggs and Lu in \cite{GL} and by Bukh \cite{B2009}, although it was already widely believed in the extremal finite set theory community. For a survey on forbidden subposet problems see \cite{GLi}.

Another basic type of extremal finite set theory problems is related to intersection patterns. We say that a family $\cF$ is \textit{intersecting} any two members of it share at least one element. Erd\H os, Ko and Rado \cite{EKR} proved that if $\cF\subset \binom{n}{k}$ is intersecting and $n\ge 2k$, then $|\cF|\le \binom{n-1}{k-1}$. For a treatment of many kind of extremal finite set theory questions, see \cite{gp}.

\medskip

A variant of the basic question arises when we are given a weight function (in addition to a property) and we want to determine the largest weight of a family satisfying the property. The most usual version is when the weight of a family is the sum of the weights of its members, and the weight of a subset of $[n]$ depends only on its size. For example the celebrated LYM inequality \cite{lubell,yam,mesh} states that for any antichain $\cF\subset 2^{[n]}$, we have $\sum_{F\in\cF} 1/\binom{n}{|F|}\le 1$.

A method to handle together all the weights of the above kind was introduced by P.L. Erd\H os, Frankl and Katona \cite{EFK}. The \textit{profile vector} of a family $\cF$ is $\underline{p}(\cF)=(f_0,\dots,f_n)$, where $f_i=|\cF\cap \binom{[n]}{i}|$. The weight vector corresponding to a weight function is $\underline{w}=(w_0,\dots,w_n)$, where $w_i$ is the weight of the $i$-element sets. Then the weight of $\cF$ is the scalar product of the profile vector and the weight vector. For a property $T$ and a positive integer $n$, there is a set of profile vectors in the $(n+1)$-dimensional Euclidean space. It is well-known that the scalar product is maximized at one of the extreme points of this set. These extreme points are the same as the extreme points of the convex hull of the set of profile vectors, which is called the \textit{profile polytope}. The extreme points of the profile polytopes have been since determined for several properties of families.

 We say that a property $T$ of families is \textit{hereditary} if for any family $\cF$ with property $T$, every subfamily of $\cF$ has property $T$. It is easy to see that a property is hereditary if and only if it can be defined by some forbidden substructures, like all the properties considered above. We remark that in case of hereditary properties, we can assume all the coordinates of weight functions are non-negative, as we could simply delete the sets of negative weights anyway. Regarding the extreme points, it means that we can obtain all the extreme points by changing to zero some coordinates of those extreme points that maximize the non-negative weight functions.

\bigskip

Forbidden subposet problems can be studied in any poset, and intersection problems can also be studied in structures other than the Boolean poset. A structure where both have been studied is the lattice of subspaces. Let $q$ be a prime power, $\mathbb{F}_q$ be a field of order $q$ and $\mathbb{F}_q^n$ be a vector space of dimension $n$ over $\mathbb{F}_q$. Let ${n \brack
k}_q=\frac{(q^n-1)(q^{n-1}-1)...(q^{n-k+1}-1)}{(q^k-1)(q^{k-1}-1)...(q-1)}$
be the Gaussian ($q$-nomial) coefficient. It is well-known that ${n\brack k}_q$ is the number of $k$-dimensional subspaces in $\mathbb{F}_q^n$. We also say that the $k$-dimensional subspaces form level $k$.

We are going to consider analogues of extremal finite set theory questions, where $i$-element subsets of $[n]$ are replaced by $i$-dimensional subspaces of $\mathbb{F}_q^n$.
We say that two subspaces \textit{intersect} if their intersection is more than just the zero vector, i.e. they share a 1-dimensional subspace. Hsieh \cite{hsieh} proved an analogue of the Erd\H os-Ko-Rado theorem by showing that an intersecting family of $k$-dimensional subspaces has cardinality at most ${n-1 \brack k-1}$, provided $n>2k$. Greene and Kleitman \cite{gk} extended it to the case $n=2k$. The analogue of Sperner's theorem is also well-known (see \cite{engel}). Profile polytopes were studied in this setting in \cite{GP2}.

Recently, other forbidden subposet problems have been examined in the poset of subspaces \cite{sash,sy}. Let $La_q(n,P)$ denote the largest number of members of a $P$-free family of subspaces of $\mathbb{F}_q^n$. Analogously to the Boolean case, we can define $e_q(P)$ to be the largest integer such that the middle $e_q(P)$ levels do not contain $P$ for any $n$, and let $\Sigma_q(n,k)=\sum_{i=1}^k{n \brack \lfloor\frac{n-k}{2}\rfloor+i}_q$. One might formulate the following.

\begin{conjecture}\label{conj2}
For any integer $n$ and poset $P$, we have
$La_q(n,P)=(1+o(1))\Sigma_q(n,e_q(P))$.
\end{conjecture}

Observe that for several posets we have $e_q(P)=e(P)$.
Rather than proving results analogous to those known in the Boolean case, the focus of the papers mentioned above is to prove ``stronger'' results. For example, the \textit{diamond} poset $D_2$ has four elements with relations $a<b<d$ and $a<c<d$. It is unknown if Conjecture \ref{conj} holds for this poset. The best upper bound is $La(n,D_2)\le (2.20711+o(1))\binom{n}{\lfloor n/2\rfloor}$ \cite{gmt}. Sarkis, Shahriari and students \cite{sash} obtained, for the analogous question in the poset of subspaces, the upper bound $(2+1/q){n\brack \lfloor n/2\rfloor}_q$.

Let $\vee$ be the poset on three elements with relations $a<b$ and $a<c$, and $\wedge$ be the poset on three elements with relations $a<c$ and $b<c$. Katona and Tarj\'an \cite{KT} determined $La(n,\{\vee,\wedge\})$, where we forbid $\vee$ and $\wedge$ at the same time. The solution is $\binom{n}{n/2}$ if $n$ is even, but slightly more than $\binom{n}{\lfloor n/2\rfloor}$ if $n$ is odd. Shahriari and Yu \cite{sy} showed that in case of subspaces, we have $La_q(n,\{\vee,\wedge\})={n\brack \lfloor n/2\rfloor}_q$ for every prime power $q$ and $n\ge 2$. They also studied the case we forbid a \textit{broom} $\wedge_u$ and a \textit{fork} $\vee_v$ at the same time, where $\wedge_u$ has $u+1$ elements $a_1,\dots, a_u,b$ and relations $a_i<b$ for any $i\le u$, while $\vee_v$ has $v+1$ elements $a,b_1,\dots,b_v$ and relations $a<b_i$ for every $i\le v$.

The \textit{butterfly} poset $B$ has four elements and relations $a<c$, $a<d$, $b<c$ and $b<d$. De Bonis, Katona and Swanepoel \cite{DKS} proved $La(n,B)=\Sigma(n,2)$. Shahriari and Yu \cite{sy} proved $La_q(n,B)=\Sigma_q(n,2)$.

In this paper we state a simple lemma (Lemma \ref{main}), that generalizes the so-called \textit{permutation method} and explore its consequences. It can be applied to other structures, and in particular for the subspaces it implies the following theorem.

\begin{thm}\label{kov}
Let $T$ be a hereditary property. If any family in $2^{[n]}$ satisfying $T$ has at most $\Sigma(n,k)$ members, then any family of subspaces of $\mathbb{F}_q^n$ with property $T$ has at most $\Sigma_q(n,k)$ members.
\end{thm}

This means that the result of  De Bonis, Katona and Swanepoel \cite{DKS} about butterflies implies the result of Shahriari and Yu \cite{sy}. Note that they also determine the extremal families. They also state a conjecture, that would follow from a result in \cite{GMNPV}, using Theorem \ref{kov}. 

The asymptotic version of Theorem \ref{kov} is also true, giving the following result.

\begin{thm}\label{kov2}
Let $T$ be a hereditary property. If any family in $2^{[n]}$ satisfying $T$ has at most $(1+o(1))\Sigma(n,k)$ members, then any family of subspaces of $\mathbb{F}_q^n$ with property $T$ has at most $(1+o(1))\Sigma_q(n,k)$ members.

\end{thm}

\begin{corollary} If Conjecture \ref{conj} holds for $P$ and $e_q(P)=e(P)$, then Conjecture \ref{conj2} also holds for $P$.

\end{corollary}

To state the covering lemma (Lemma \ref{main}), we need some preparation, hence we postpone it to Section 2. We also describe how it relates to several known proofs. In Section 3 we prove Theorems \ref{kov} and \ref{kov2}. In Section 4 we examine its relation to profile polytopes and related topics, and initiate the study of generalized forbidden subposet problems in the poset of subspaces.

\section{The main lemma}

Our lemma is motivated by a lemma by Holroyd and Talbot \cite{ht}.  
We say that a family of subsets of $S$ is \textit{$t$-covering} if every element of $S$ is contained in exactly $t$ sets of the family. Given a partition of $S$ into $S_0\cup S_1\cup \dots \cup S_n$ and a vector $\underline{t}=(t_0,t_1,\dots, t_n)$, we say that a family of subsets of $S$ is a \textit{$\underline{t}$-covering} of $S$ if for each $0\le i\le n$, every element of $S_i$ is contained in exactly $t_i$ sets of the family. 

In our applications, $S$ will be $2^{[n]}$ or the family of subspaces of $\mathbb{F}_q^n$, and $S_i$ will be level $i$.
Holroyd and Talbot \cite{ht} considered coverings of subfamilies $\cF$ of one level $\binom{[n]}{i}$, and their lemma stated that if an element $x$ has the property that the largest intersecting family in every $\cG\in\Gamma$ is $\{G\in \cG: x\in G\}$, then the largest intersecting family in $\cF$ is $\{F\in \cF: x\in F\}$. Our main contribution is the simple observation that we can extend their method to other forbidden configurations and more levels.

For a weight vector $\underline{w}=(w_0,\dots,w_n)$ and a set $F\subset S$, let $\underline{w}(F)=\sum_{i=0}^n w_i|F\cap S_i|$. Let $\underline{w/t}=(w_0/t_0,\dots,w_n/t_n)$. We will always assume that every coordinate of every weight vector is non-negative. 
A version of the lemma below has already appeared in my master's thesis \cite{gerb}.

\begin{lemma}[Covering lemma]\label{main}
Let $T$ be a hereditary property of subsets of $S$ and $\Gamma$ be a $\underline{t}$-covering family of $S$. Assume that for every $G\in\Gamma$, every subset $G'$ of $G$ with property $T$ has $\underline{w/t}(G')\le x$. Then $\underline{w}(F)\le |\Gamma|x$ for every $F\subset S$ with property $T$.
\end{lemma}

\begin{proof}

Let $F$ be a set with property $T$.

Observe that we have $t_i|F\cap S_i|=\sum_{G\in\Gamma} |G\cap F\cap S_i|$, as every element of $F\cap S_i$ is counted $t_i$ times on both sides. Thus we have

\begin{equation*}
    \begin{split}
        \underline{w}(F)=&\sum_{i=0}^n w_i|F\cap S_i|=\sum_{i=0}^n \frac{w_i}{t_i}t_i|F\cap S_i|=\sum_{i=0}^n \frac{w_i}{t_i}\sum_{G\in\Gamma} |G\cap F\cap S_i|\\
        =&\sum_{G\in\Gamma}\sum_{i=0}^n \frac{w_i}{t_i}|G\cap F\cap S_i|=\sum_{G\in\Gamma}\underline{w/t}(G\cap F)\le \sum_{\cG\in\Gamma} x=|\Gamma|x.
    \end{split}
\end{equation*}
\end{proof}

Let us describe how one can use this lemma in extremal finite set theory. Let $S=2^{[n]}$ and $S_i=\binom{[n]}{i}$. Then subsets of $S$ are families, we will denote them by $\cF$ and $\cG$ instead of $F$ and $G$.

The prime examples of covering families where the above lemma is useful are given by the \textit{permutation method}. Given a permutation $\alpha:[n]\rightarrow [n]$, and a set $F\subset [n]$, let $\alpha(F)=\{\alpha(i): i\in F\}$. Similarly, for a family $\cF\subset 2^{[n]}$, let $\alpha(\cF)=\{\alpha(F):F\in \cF\}$. 

Let $\cG_0$ be a family that has at least one $i$-element set for every $0\le i\le n$, and let $\Gamma$ consist of $\alpha(\cG_0)$ for all permutations $\alpha$. Let $g_i=|\cG_0\cap \binom{[n]}{i}|>0$ and $t_i=g_ii!(n-i)!$, then $\Gamma$ is a $\underline{t}$-covering of $2^{[n]}$. 

The simplest example is when $\cG_0$ is a full chain. Consider a Sperner family  $\cF\subset 2^{[n]}$ and let $\underline{w}=\underline{t}$. Then $\sum_{F\in \cF}|F|!(n-|F|)!=\underline{w}(\cF)=\sum_{\cG\in\Gamma}\sum_{H\in \cG\cap \cF}\underline{w/t}(H)\le \sum_{\cG\in\Gamma} 1=|\Gamma|=n!$. Dividing by $n!$ we obtain the already mentioned LYM-inequality. Another example is when $\cG_0$ is the family of intervals in a cyclic ordering of $[n]$, resulting in the cycle method \cite{kat}.

Any family $\cG_0$ can be used to give upper bounds on problems in extremal finite set theory, but these bounds are unlikely to be sharp. For that, $\cG_0$ has to be very symmetric in a sense. We need that for \emph{every} permutation $\alpha$, the largest subfamily of $\alpha(\cG_0)$ with property $T$ has the same size.
Other examples for families $\cG_0$ that sometimes give sharp bounds are the chain-pairs \cite{gerbner} and double chains \cite{BN}. 

\medskip

Let us return to Lemma \ref{main} and examine a very special case. Assume $S_{i_1}\cup \dots \cup S_{i_k}$ has property $T$ and for every $G\in\Gamma$, $\underline{w/t}(G')= x$ for $G'=G\cap (S_{i_1}\cup \dots \cup S_{i_k})$ (In case of the permutation method, it means that the union of $k$ full levels have property $T$, and the weight inside $\alpha(\cG_0)$ is maximized by those $k$ levels). This implies that we have equality in Lemma \ref{main}. 

Now assume that we conjecture that $\underline{w}(\cF)$ is maximized by a family that is the union of $k$ full levels (among families with property $T$). Let $\cH_0$ be the intersection of those $k$ levels with $\cG_0$, then $\cH_0$ has property $T$. If $\cH_0$ happens to have the largest weight $\underline{w/t}$ among subfamilies of $\cG_0$ with property $T$, then it proves the conjecture (here we use the simple observation that $\alpha(\cH_0)$ would maximize $\underline{w/t}$ among subfamilies of $\alpha(\cG_0)$). Thus our goal would be to find $\cG_0$ with this property. 

For example, in case of antichains, it is a natural idea to consider a full chain as $\cG_0$. Indeed, for every weight, the maximum will be given by a family that consists of one member, which is a full level on the chain. Moreover, it is one of the levels with the largest weight, thus we can choose the same level all the time. This implies that for every weight function, the maximum in the Boolean poset is also given by a full level, giving us not only Sperner's theorem and the LYM inequality, but all the extreme points of the profile polytope, reproving a result in \cite{EFK}. Moreover, we say that a family is \textit{$k$-Sperner} if it is $P_{k+1}$-free. The above argument works for $k$-Sperner families as well, since on any chain, for any weight, the maximum is given by $k$ full levels. This, again, gives the extreme points of the profile polytope as well, reproving a result in \cite{EFK2}.


Observe that we do not need to have full levels in our conjecture to obtain an exact result without further computations. Assume that in our conjecture, for every $i$, the extremal family $\cH$ contains $\gamma_i\binom{n}{i}$ sets from level $i$, and $\cH$ contains a $\gamma_i$ fraction of the intersection of $\alpha(\cG_0)$ and level $i$. Then the same argument works. For example consider intersecting families on level $k$, and use the cycle method \cite{kat}. We choose a cyclic ordering of the elements of $[n]$ and let $\cG_0$ be the family of $k$-sets of consecutive elements. There are $n$ such $k$-sets, and $k$ of them contains a fixed element $x$. Let $\cH$ be the family of $k$-sets containing $x$, and $\cH_0$ be its intersection with $\cG_0$. It is not hard to see that $\cH_0$ is the largest intersecting subfamily of $\cG_0$ (provided $k\le n/2$). Thus, for every $\alpha$ we have that $\cH$ contains a $k/n$ fraction of the members of $\alpha(\cG_0)$. As $\cH$ contains a $k/n$ fraction of all the sets, we are done.

To finish this section, let us remark that we are mostly interested in the case every $w_i=1$. For that $w_i/t_i=1/(g_i i!(n-i)!)=\binom{n}{i}/n!g_i$. In case $\cG_0$ is a full chain, every $g_i$ is the same, in case $\cG_0$ is the cycle, almost every $g_i$ is the same (with the exception of $g_0$ and $g_n$). As multiplying with the same number does not change the extremal families, we can consider maximizing the weight function with $w'_i=\binom{n}{i}$ instead (assuming we can deal with the empty set and the full set some other way). If, on the other hand we can deal with the case of constant weight on the chain or the cycle for a property $T$, and the optimal family consists of the middle levels, then we obtain a LYM-type inequality for subfamilies of $2^{[n]}$ with property $T$, see for example the case of butterfly-free families in \cite{DKS}.










\section{Subspaces}

Let us turn our attention to $q$-analogues. Similarly to the Boolean case and the permutation method, it will again simplify our tasks if all $\cG\in \Gamma$ are isomorphic. Moreover, we would prefer to use $\cG$ where proving extremal results is either easy or has already been done. Therefore, we will use $\cG=2^{[n]}$. Indeed, it is isomorphic to a subfamily of subspaces. Choose an arbitrary basis $v_1,\dots,v_n$ of $\mathbb{F}_q^n$, and let $\cG$ be the family subspaces generated by a set of these vectors. Obviously the function that maps $H\subset [n]$ to the subspace $\langle v_x:x\in H\rangle$ keeps inclusion and intersection properties.

There are  $f(q,n)=(q^n-1)(q^n-q)(q^n-q^2)\cdots (q^n-q^{n-1})/n!$ ways to choose a basis, as we pick the vectors one by one, and we obtain a basis $n!$ ways. Hence $f(q,n)$ is the cardinality of $\Gamma$, which is a $\underline{t}$-covering of the subspaces of $\mathbb{F}_q^n$ with $t_i=\frac{(q^{i}-1)\cdots (q^{i}-q^{i-1})(q^n-q^{i})\cdots (q^n-q^{n-1})}{i!(n-i)!}$. Indeed, to count how many times an $i$-dimensional subspace is covered, we have to pick a basis of the $i$-dimensional subspace first, and then extend it to a basis of $\mathbb{F}_q^n$. We counted every $\cG\in\Gamma$ exactly $i!(n-i)!$ times, as we picked the basis in an ordered way. Observe that we have $t_0>t_1>\dots>t_{\lfloor n/2\rfloor}=t_{\lceil n/2\rceil}<t_{\lceil n/2\rceil+1}<\dots t_n$. 

Now we are ready to prove Theorem \ref{kov}, which states that if every family $\cF\subset 2^{[n]}$ satisfying a hereditary property $T$ has cardinality at most $\Sigma(n,k)$, then families of subspaces of $\mathbb{F}_q^n$ with property $T$ have cardinality at most $\Sigma_q(n,k)$. We note that the actual calculation could be omitted by the arguments presented in Section 2. We include it here for sake of completeness.

\begin{proof}[Proof of  Theorem \ref{kov}] Let $\cF$ be a family of subspaces satisfying $T$. Consider the $\underline{t}$-covering family $\Gamma$ defined above and let $w_i=t_i$. Then every $\cG\in \Gamma$ is isomorphic to $2^{[n]}$, thus by our assumption, the largest weight $\underline{w/t}$, i.e. the largest cardinality of a subfamily $\cG'\subset \cG$ satisfying $T$ is $\Sigma(n,k)$. This implies $\underline{w}(\cF)\le |\Gamma|\Sigma(n,k)$. To maximize $|\cF|$ among those families satisfying the above inequality, we need to pick subspaces with the smallest weight, i.e. from the middle levels. We claim that we can pick exactly the $k$ full middle levels, i.e. $\underline{w}(\cF_0)= |\Gamma|\Sigma(n,k)$ for the family $\cF_0$ consisting of $k$ middle levels. (Note that if $n+k$ is even, we have two options for $\cF_0$). This will finish the proof, because more than $\Sigma_q(n,k)$ subspaces would have larger weight than $|\Gamma|\Sigma(n,k)$.

We have
\begingroup\makeatletter\def\f@size{9.9}\check@mathfonts
\begin{equation*}
    \begin{split}
        \underline{w}(\cF_0)=&\sum_{i=\lfloor\frac{n-k}{2}\rfloor+1}^{\lfloor\frac{n-k}{2}\rfloor+k} w_i{n \brack i}_q=\sum_{i=\lfloor\frac{n-k}{2}\rfloor+1}^{\lfloor\frac{n-k}{2}\rfloor+k}\frac{(q^{i}-1)\cdots (q^{i}-q^{i-1})(q^n-q^{i})\cdots (q^n-q^{n-1})}{i!(n-i)!}{n\brack i}_q\\
        =&\sum_{i=\lfloor\frac{n-k}{2}\rfloor+1}^{\lfloor\frac{n-k}{2}\rfloor+k}\frac{(q^{i}-1)\cdots (q^{i}-q^{i-1})(q^n-q^{i})\cdots (q^n-q^{n-1})}{i!(n-i)!}\frac{(q^n-1)\dots (q^n-q^{n-1})}{(q^i-1)\dots (q^{i}-q^{i-1})(q^{n-i}-1)\dots (q^{n-i}-q^{n-i-1})}\\
        =&\sum_{i=\lfloor\frac{n-k}{2}\rfloor+1}^{\lfloor\frac{n-k}{2}\rfloor+k} \frac{f(q,n) n!}{i!(n-i)!}=\sum_{i=\lfloor\frac{n-k}{2}\rfloor+1}^{\lfloor\frac{n-k}{2}\rfloor+k} |\Gamma|\binom{n}{i}=|\Gamma|\Sigma(n,k).
    \end{split}
\end{equation*}\endgroup
\end{proof}

Note that there are several similar statements we could prove. We chose to state this one because it immediately gives the exact value of $La_q(n,B)$. Observe that the Boolean result actually gives a weighted result in the case of subspaces, that is stronger than Theorem \ref{kov}. In case of the butterfly poset, we could even use the LYM-type inequality in the Boolean poset, to obtain that for a butterfly-free family $\cF$ of subspaces, we have $\sum_{F\in\cF}1/{n\brack dim(F)}\le 2$. Let us prove now Theorem \ref{kov2}, which is the asymptotic version of Theorem \ref{kov}.

\begin{proof}[Proof of Theorem \ref{kov2}.] Let $P$ be an arbitrary poset, assume Conjecture \ref{conj} holds, and let $\cF$ be a $P$-free family of subspaces of $\mathbb{F}_q^n$. We follow the proof of Theorem \ref{kov}. Using its notation, we obtain $\underline{w}(\cF)\le (1+o(1))|\Gamma|\Sigma(n,k)$. Again, to maximize $|\cF|$ among those families satisfying the above inequality, we need to pick subspaces with the smallest weight, i.e. from the middle levels. This time we claim that we can pick the subspaces in $\cF_0$, and $o(|\cF_0|$ additional subspaces. This will finish the proof similarly to the proof of Theorem \ref{kov}.

We have proved $\underline{w}(\cF_0)= |\Gamma|\Sigma(n,k)$, thus the remaining subspaces have total weight $o(|\Gamma|\Sigma(n,k))=o(\underline{w}(\cF_0))$. As each of those have weight not smaller than any weight in $\cF_0$, more than $\varepsilon|\cF_0|$ of them would have weight more than $\varepsilon \underline{w}(\cF_0)$, a contradiction that finishes the proof.
\end{proof}



\section{Profile polytopes, chain profile polytopes, generalized forbidden subposet problems}

In the previous sections we considered arbitrary weights. This means our method can potentially determine the extreme points of the profile polytope for a hereditary property $T$. If every extreme point in the Boolean case is the union of full levels, and the corresponding union of full levels has property $T$ in the case of subspaces, then this is the situation. Unfortunately, we are only aware of one particular property where this is the situation. For $k$-Sperner families, the Boolean result was proved in \cite{EFK2}. We note that instead of using the substructure isomorphic to $2^{[n]}$ with Lemma \ref{main}, one could use a simpler substructure: a full chain with Lemma \ref{main}, to obtain the same result, i.e. to determine the extreme points. Moreover, it also easily follows from the LYM-inequality, which is known to hold for the poset of subspaces. In fact, one can analogously define the profile vectors and polytopes for any graded poset and show for a large class of posets that the extreme points of $k$-Sperner families are the profiles of the unions of at most $k$ full levels.

Gerbner and Patk\'os \cite{GP} introduced $l$-chain profile vectors. Given a family $\cF$, its $l$-chain profile vector is an element of the $\binom{n+1}{l}$-dimensional Euclidean space. A coordinate corresponds to a set $\{i_1,\dots,i_l\}$ with $i_1<i_2<\dots<i_l$. The value of that coordinate is the number of chains of length $l$ with one element from level $i_j$ for every $1\le j\le l$. They determined the extreme points of the $l$-chain profile polytopes of intersecting families and of $k$-Sperner families.

They mentioned in \cite{GP2}, after determining the extreme points of the profile polytope of intersecting families of subspaces, that with the same method, one can determine the extreme points of the $l$-chain profile polytope as well. Here we show that similarly, the extreme points of the $l$-chain profile polytope of $k$-Sperner families of subspaces can be determined. We will state a modified version  of Lemma \ref{main} that counts copies of a poset $Q$ instead of elements.

Let $Q$ be an arbitrary poset with elements $a_1,\dots, a_l$. Consider the $r=l^{n+1}$ functions that map every $a_j$ to an $S_i$.
Let us fix an ordering of these functions and let $f_i$ be the $i$th of them. For each $1\le i\le r$, let $\cS_i$ be an arbitrary family of $l$-sets with one element in $f_i(a_j)$ for every $1\le j\le l$. In the applications, where $S_i$ is a level, we will let $\cS_i$ consist of those $l$-sets, where the elements form a copy of $Q$. In particular, if for an embedding $f_i$ and for some $j,j'$ with $a_j<a_j'$ we have $f(a_j)$ is higher level than $f(a_{j'})$, then $\cS_i$ is empty. Let us consider only those $r'\le r$ functions $f_i$, where $\cS_i$ is not empty. We can assume without loss of generality that these functions are $f_1,\dots,f_{r'}$. 

Let $\underline{t}=(t_1,\dots,t_{r'})$ be a vector.
We say that a family $\Gamma$ of subsets of $S$ is \textit{$(l,\underline{t})$-covering} if for each $1\le i\le r'$, and each $l$-set in $\cS_i$, there are $t_i$ members of $\Gamma$ containing all the elements of that $l$-set (i.e. a particular copy of $Q$).
Let us consider a weight vector $\underline{w}=(w_1,\dots,w_{r'})$. For a set $F\subset S$, let $f_i$ denote the number of $l$-sets in $\cS_i$ with every element in $F$. Let $\underline{w}(F)=\sum_{i=1}^{r'} w_if_i$. Let $\underline{w/t}=(w_1/t_1,\dots,w_{r'}/t_{r'})$. We will assume that every weight is non-negative (as $T$ is hereditary, elements of $S$ with negative weight could simply be deleted anyway from any subset of $S$ with property $T$).

\begin{lemma}\label{mainl} Let $T$ be a hereditary property of subsets of $S$ and $\Gamma$ be an $(l,\underline{t})$-covering family of $S$. Assume that for every $G\in \Gamma$, every subset $G'$ of $G$ with property $T$ has $\underline{w/t}(G')\le x$. Then $\underline{w}(F)\le |\Gamma|x$ for every $F\subset S$ with property $T$.

\end{lemma}

\begin{proof} Observe that we have $t_if_i=\sum_{G\in\Gamma}h_i$, where $h_i$ denotes the number of $l$-sets in $\cS_i$ with each element of it in $F\cap G$. Indeed, the $l$-sets in $\cS_i$ with each element in $F$ are counted $t_i$ times on both sides. Thus we have

\begin{equation*}
    \begin{split}
        \underline{w}(F)=&\sum_{i=1}^{r'} w_if_i=\sum_{i=1}^{r'} \frac{w_i}{t_i}t_if_i=\sum_{i=1}^{r'} \frac{w_i}{t_i}\sum_{G\in\Gamma}h_i=\sum_{G\in\Gamma}\sum_{i=1}^{r'} \frac{w_i}{t_i}h_i\\
        =&\sum_{G\in\Gamma}\underline{w/t}(G\cap F)\le \sum_{G\in\Gamma} x=|\Gamma|x.
    \end{split}
\end{equation*}
\end{proof}

We have equality here if for every $G\in\Gamma$, there is a $G'\subset G$ satsifying $T$ with $\underline{w/t}(G')=x$, and $G'=G\cap F$. This holds in the following situation. Let $T$ be the $k$-Sperner property, $S$ be the family of subspaces of $\mathbb{F}_q^n$ with the usual partition into levels, and $\cS_i$ be those $l$-sets that form a chain. Let $\Gamma$ consist of copies of the Boolean poset, as described in Section 3 (note that we could use instead the chains given by a basis and its ordering). Let us assume levels $j_1,\dots,j_k$ have the maximum weight $\underline{w/t}$ in the Boolean poset, and let $F$ consist of the subspaces on levels $j_1,\dots,j_k$. Then by the above, $F$ has the largest weight $\underline{w}(F)=|\Gamma|x$ among $k$-Sperner families. We obtained that for every non-negative weight the union of $k$ levels have the largest weight, which implies the following result.

\begin{corollary}\label{trivi} The extreme points of the $l$-chain profile polytope of $k$-Sperner families of subspaces of $\mathbb{F}_q^n$ are the unions of at most $k$ levels.

\end{corollary}

We mentioned the $l$-chain polytopes here because the above result gives the first instance of a generalized forbidden subposet problem in the poset of subspaces. The generalized forbidden subposet problem seeks to find $La(n,P,Q)$, the largest number of copies of the poset $Q$ in a $P$-free subfamily of $2^{[n]}$. Its study was initiated by Gerbner, Keszegh and Patk\'os \cite{GKP}, analogously to the graph case \cite{AS} that has recently attracted a lot of attention. Further results on $La(n,P,P_l)$ can be found in \cite{GMNPV2}. 

On the one hand, Lemma \ref{mainl} shows that studying weighted versions of this problem on the cycle can potentially help obtain bounds. However, in case of counting the members of a family, we had the useful property that $w_i/t_i$ is the largest in the middle, exactly where the (conjectured) extremal families are. Therefore, an unweighted result on the cycle gave a weighted result in the Boolean case that implied the unweighted result. And similarly, an unweighted result for the Boolean case immediately implied the analogous bound for $La_q(n,P)$. However, this is not the case with the more complicated weight functions and more diverse extremal families that we deal with in generalized forbidden subposet problems.

On the other hand, we propose to study generalized forbidden subposet problems in the poset of subspaces, and let $La_q(n,P,Q)$ denote the largest number of copies of the poset $Q$ in a $P$-free family of subspaces of $\mathbb{F}_q^n$. Corollary \ref{trivi} implies that $La_q(n,P_k,P_l)$ is given by $k$ full levels (it is not hard to see that the best way to choose the $k$ levels $i_1,\dots,i_k$ is when the values $i_1,i_2-i_1,i_3-i_2,\dots,i_k-i_{k-1},n-i_k$ differ by at most one). For other pairs of posets, a weighted version in the Boolean case could give bounds on $La_q(n,P,Q)$.

To finish the paper, we obtain some simple results for $La_q(n,P,Q)$. They are unrelated to the earlier parts of the paper, but we would like to present some results concerning the topic we initiate the study of. Let the \textit{generalized diamond} poset $D_r$ have $r+2$ elements $a,b_1,\dots,b_r,c$ and relations $a<b_i<c$ for $1\le i\le r$.

\begin{prop}
\label{thm:easy}
\textbf{(i)} $La_q(n,\vee, {\wedge}_r)=La(n,{\wedge},\vee_r)={{n \brack \lfloor n/2\rfloor}_q \choose r}$.

\textbf{(ii)} $La_q(n,B,D_r)={{n \brack\lfloor n/2\rfloor}_q\choose r}$.

\textbf{(iii)} $La_q(n,P_r,\wedge_r)=\max_{0\le k\le n} {n\brack k}\binom{{k\brack \lfloor k/2\rfloor}_q}{r}$.

\end{prop}

The Boolean analogues of the above statements were proved in \cite{GKP}, and the proofs of them also work in our case. We include them for sake of completeness. We will use the \textit{canonical partition} of $k$-Sperner families $\cF$; it is a partition into $k$ antichains $\cF_1,\dots, \cF_k$, where $\cF_i$ is the set of minimal elements of $\cF\setminus \cup_{j=1}^{i-1} \cF_j$.

\begin{proof} The lower bounds for \textbf{(i)} and \textbf{(ii)} are given by the families consisting of all the $\lfloor n/2\rfloor$-dimensional subspaces together with the zero-dimensional and/or the $n$-dimensional subspace. For \textbf{(iii)} consider all the $k$-dimensional and $\lfloor k/2\rfloor$-dimensional subspaces for every $k$.

For the upper bound in $\textbf{(i)}$, the first equality is trivial by symmetry. Let us consider now the canonical partition $\cF_1\cup \cF_2$ of a $\vee$-free family $\cF$. Observe that every copy of $\wedge_r$ consists of a member of $\cF_2$, and $r$ members of $\cF_1$ contained in it. Every member of $\cF_1$ is contained in at most one member of $\cF_2$ by the $\vee$-free property, thus for every set of $r$ members of $\cF_1$, at most one member of $\cF_2$ forms a copy of $\wedge_r$ with them. This implies  $La_q(n,\vee, {\wedge}_r)\le \binom{|\cF_1|}{r}$. As $\cF_1$ is an antichain, it has at most ${n \brack \lfloor n/2\rfloor}_q$ members, finishing the proof of  \textbf{(i)}. 

To prove the upper bound in \textbf{(ii)}, let $\cF$ be a $B$-free family of subspaces and $\cM=\{M\in\cF: \exists F',F''\in \cF \,\,\textrm{such that}\, F'< M< F''\}$. As $\cF$ is $P_4$-free, $\cM$ is an antichain. Observe that for an $M\in \cM$ there is exactly one $F'\in \cF$ with $F'<M$ and there is exactly one $F''\in \cF$ with $M<F''$. Thus, for every $r$-tuple from $\cM$ there is at most one copy of $D_r$ in $\cF$, and there are at most ${{n \brack\lfloor n/2\rfloor}_q\choose r}$ such $r$-tuples.

To prove the upper bound in \textbf{(iii)}, let $\cF$ be a $P_3$-free family of subspaces and consider its canonical partition $\cF_1\cup\cF_2$. Every copy of $\wedge_r$ consists of a member of $\cF_2$ and $r$ members of $\cF_1$. For a member $F$ of $\cF_2$ with dimension $k$, we have to pick $r$ subspaces of it that are in $\cF_1$. Those members of $\cF_1$ that can be picked form an antichain of subspaces of a $k$-dimensional space, thus there are at most ${i\brack \lfloor k/2\rfloor}_q$ of them, and there are ${n\brack k}\binom{{k\brack \lfloor k/2\rfloor}_q}{r}$ ways to pick $r$ of them. It means that a $k$-dimensional member of $\cF_2$ is in at most $w(k):=\binom{{k\brack \lfloor k/2\rfloor}_q}{r}$ copies of $\wedge_r$. Hence the total number of copies of $\wedge_r$ is at most the total weight of $\cF_2$, i.e. $w(\cF_2)$. As $\cF_2$ is an antichain, this is maximized by a level (for a number of reasons mentioned earlier, for example Corollary \ref{trivi} implies this). The weight of level $k$ is ${n\brack k}\binom{{k\brack \lfloor k/2\rfloor}_q}{r}$, finishing the proof.

\end{proof}


\begin{thebibliography}{}

\bibitem{AS}
\textsc{N. Alon, C. Shikhelman}, Many $T$-copies in $H$-free graphs, J. Combinatorial Theory, Ser. B, 121 (2016), 146--172.

 \bibitem{B2009}
\textsc{B. Bukh}, Set families with a forbidden subposet, Electronic Journal of Combinatorics, 16(1) (2009) P142.

 \bibitem{BN}
 \textsc{P. Burcsi, D. Nagy}, The method of double chains for largest families with excluded subposets, Electronic Journal of Graph Theory and Applications, 1 (2013) 40--49.
 
\bibitem{DKS}
\textsc{A DeBonis, G.O.H. Katona, K. Swanepoel}, Largest family without $A\cup B \subset C\cap D$, J. Combin. Theory Ser. A 111 (2005) 331--336.

\bibitem{engel} \textsc{K. Engel}, Sperner Theory, Encyclopedia of Mathematics and its Applications,
65. Cambridge University Press, Cambridge, 1997. x+417 pp.



\bibitem{EKR} \textsc{P. Erd\H{o}s, Chao Ko, R. Rado}, Intersection theorems for systems of finite sets. \textit{Quart. J. Math. Oxford}, \textbf{12}, 313--320, 1961.

\bibitem{EFK}
\textsc{P.L. Erd\H os, P. Frankl, G.O.H. Katona}, Intersecting Sperner families and their convex hulls, Combinatorica 4 (1984) 21--34.

\bibitem{EFK2}
\textsc{P.L. Erd\H os, P. Frankl, G.O.H. Katona}, Extremal hypergraph problems and convex hulls. Combinatorica, 5(1), 11--26.

\bibitem{gerb} \textsc{D. Gerbner}, A k\"orm\'odszer az extrem\'alis kombinatorik\'aban (in Hungarian), Master's thesis, ELTE, Budapest, Hungary, 2004.

\bibitem{gerbner} \textsc{D. Gerbner}, Profile polytopes of some classes of families, Combinatorica 33 (2013) 199--216.

\bibitem{GKP}
\textsc{D. Gerbner, B. Keszegh, B. Patk\'os.} Generalized forbidden subposet problems, arxiv:1701.05030

\bibitem{GMNPV} \textsc{D. Gerbner, A. Methuku, D.T. Nagy, B. Patk\'os, M. Vizer} (2017). Forbidding rank-preserving copies of a poset. Order, 1-10.

\bibitem{GMNPV2} \textsc{D. Gerbner, A. Methuku, D.T. Nagy, B. Patk\'os, M. Vizer}, On the number of containments in P-free families, arXiv:1804.01606

\bibitem{GP}
\textsc{D. Gerbner, B. Patk\'os}, $l$-chain profile vectors, SIAM J. Discrete Math. 22 (2008) 185--193. 

\bibitem{GP2}
\textsc{D. Gerbner, B. Patk\'os}, (2009). Profile vectors in the lattice of subspaces. Discrete Mathematics, 309(9), 2861--2869.

\bibitem{gp} \textsc{D. Gerbner, B. Patk\'os}, Extremal Finite Set Theory,
1st Edition, CRC Press, 2018.

\bibitem{gk} \textsc{C. Greene, D.J. Kleitman}, Proof techniques in the theory of finite sets.
Studies in combinatorics, pp. 22--79 MAA Stud. Math., 17, Math. Assoc. America,
Washington, D.C., 1978.


\bibitem{GLi}
\textsc{J.R. Griggs and W.-T. Li}, Progress on poset-free families of subsets, in: Recent Trends in Combinatorics, 2016, 317--338.

\bibitem{GL}
\textsc{J. R. Griggs and L. Lu}, On families of subsets with a forbidden subposet, J. Combinatorial Theory (Ser. A) 119 (2012), 310–-322.

\bibitem{gmt} \textsc{D. Gr\'osz, A. Methuku, C. Tompkins}, (2018). An upper bound on the size of diamond-free families of sets. Journal of Combinatorial Theory, Series A, 156, 164--194.

\bibitem{hsieh} \textsc{W.N. Hsieh}, Intersection theorems for systems of finite vector spaces. Discrete
Math. 12 (1975), 1--16.

\bibitem{ht} \textsc{F. Holroyd, J. Talbot}, (2005). Graphs with the Erd\H os--Ko--Rado property. Discrete mathematics, 293(1--3), 165--176.

\bibitem{kat} \textsc{G. O. H. Katona}, A simple proof of the Erd\H os--Chao Ko--Rado theorem, J. Combinatorial
Theory Ser. B 13 (1972), 183--184.

\bibitem{KT}
\textsc{G.O.H. Katona and T. Tarj\'an,} Extremal problems with excluded subgraphs in the $n$-cube, Graph Theory, Lagow, 1981, Lecture Notes in Math. 1018 (Springer-Verlag, Berlin, 1983) 84--93.

\bibitem{lubell} \textsc{D. Lubell}, A short proof of Sperner's lemma, J. Combinatorial Theory 1(2) (1966), 299.

\bibitem{mesh} \textsc{L.D. Meshalkin}, A generalization of Sperner’s theorem on the number of subsets of a finite set, Teor.
Verojatnost. i Primen. 8 (1963) 219--220 (in Russian with German summary).

\bibitem{sash} \textsc{G. Sarkis, S. Shahriari, PCURC} (2014). Diamond-free subsets in the linear lattices. Order, 31(3), 421--433. 

\bibitem{sy} \textsc{S. Shahriari, S.  Yu}, (2018). Avoiding Brooms, Forks, and Butterflies in the Linear Lattices. arXiv preprint arXiv:1807.06259.

\bibitem{S}
\textsc{E. Sperner}, Ein Satz \"uber Untermengen einer endlichen Menge, Mathematische Zeitschrift 27 (1928), 544--548.

\bibitem{yam}
\textsc{K. Yamamoto}, Logarithmic order of free distributive lattices, J. Math. Soc. Japan 6 (1954) 347--357.

\end{thebibliography}
\end{document}